\documentclass{amsart}
\usepackage{amsthm,enumerate}
\usepackage{amsmath,amssymb}
\newtheorem{theorem}{Theorem}

\newtheorem{lem}[theorem]{Lemma}
\newtheorem{definition}[theorem]{Definition}

\numberwithin{theorem}{section}
\newtheorem{rem}[theorem]{Remark}
\allowdisplaybreaks
\newcommand{\Om} {\Omega}

\newcommand{\be} {\begin{equation}}
\newcommand{\ee} {\end{equation}}
\newcommand{\bea} {\begin{eqnarray}}
\newcommand{\eea} {\end{eqnarray}}
\newcommand{\Bea} {\begin{eqnarray*}}
\newcommand{\Eea} {\end{eqnarray*}}

\newcommand{\De} {\Delta}
\newcommand{\la} {\lambda}
\newcommand{\h}{\mathbb{H}^n}
\newcommand{\Dh}{\Delta_{\mathbb{H}^n}}
\newcommand{\nh}{\nabla_{\mathbb{H}^n}}
\def\R{{\mathbb R}}
\def\H{{\mathbb H}}

\newcommand{\norm}[1]{\left\Vert#1\right\Vert}
\newcommand{\rar}{\rightarrow}
\numberwithin{equation}{section}
\begin{document}
\title[Stability of positive solution ]{Stability of positive solution to biharmonic equations with logistic-type nonlinearities on Heisenberg group}
\author[G.Dwivedi,\, J.\,Tyagi ]
{ G. Dwivedi,  J.Tyagi }

\address{G.\,Dwivedi \hfill\break
 Birla Institute of Technology and Science Pilani, Pilani Campus \newline
 Pilani, Jhunjhunu \newline
 Rajasthan, India - 333031}
 \email{gaurav.dwivedi@pilani.bits-pilani.ac.in}

\address{J.\,Tyagi \hfill\break
 Indian Institute of Technology Gandhinagar \newline
 Palaj, Gandhinagar \newline
 Gujarat, India - 382355}
 \email{jtyagi@iitgn.ac.in, jtyagi1@gmail.com}

\thanks{Submitted 27--04--2019.  Published-----.}
\subjclass[2010]{Primary 35B35;  Secondary 35B09, 35J91, 35R03.}
\keywords{bi-Laplacian;\,semi-stability;\,positive solution;\,Heisenberg group.}
\begin{abstract}
In this note, we establish the existence of a positive solution and its semi-stability to the following class of biharmonic problems with logistic-type nonlinearities
\begin{equation}\label{ab1}
\left\{
  \begin{array}{ll}
    \Delta_{\mathbb{H}^n}^2u=a(\xi)u-f(\xi,u)\,\,\,\,\text{in }\Omega\\
    u|_{\partial\Omega}=0=\left.\Delta_{\mathbb{H}^n} u\right|_{\partial\Omega},
 \end{array}
\right.
\end{equation}
where $\Omega\subset\mathbb{H}^n$ is an open, smooth and bounded subset of Heisenberg group $\H^n.$ 
We establish the existence of a solution by Schauder's fixed point theorem and then with the aid of strong maximum principle, 
we obtain the positivity of the solution.
We also show that the principal eigenvalue of the linearized equation
associated with \eqref{ab1} is non-negative and hence the solution $u$ of \eqref{ab1} is semi-stable. This is shown by testing the equation under consideration with a suitable 
test function.

\end{abstract}
\maketitle

\section{Introduction}
The aim of this note is to establish the existence of a solution and its semi-stability to the following biharmonic problem on Heisenberg group:
\begin{equation}\label{pm}
\left\{
  \begin{array}{ll}
    \Delta_{\mathbb{H}^n}^2u=a(\xi)u-f(\xi,u)\,\,\,\,\text{in }\Omega\\
    u>0 \,\,\,\,\,\, \text{ in } \Omega\\
   u=0=\Delta u \,\,\,\,\, \text{on } \partial\Omega,
 \end{array}
\right.
\end{equation}
where $\Omega\subset\mathbb{H}^n$ is an open, smooth and bounded subset of Heisenberg group $\H^n,$ $a\in L^{\infty}(\Omega)$ and $f\in C(\overline{\Om}\times \R,\,\R).$ 

The functional associated with \eqref{pm} is
$$E\colon D^{2,2}(\Omega)\cap D_0^{1, 2}(\Omega)  \longrightarrow \mathbb{R}$$
defined by
\[E(u)=\frac{1}{2}\int_\Omega|{\Delta_{\mathbb{H}^n}} u|^2 d\xi-\frac{1}{2}\int_\Omega a(\xi)u^2 d\xi +\int_\Omega F(\xi,\,u)d\xi,\]
where
\[F(\xi,s)=\int_0^s f(\xi,t)dt\]
and the spaces $D^{2,2}(\Omega)$  and  $D_0^{1, 2}(\Omega)$ are defined later.
The weak formulation of \eqref{pm} is the following:
\begin{equation}\label{we}
\int_\Omega {\Delta_{\mathbb{H}^n}} u{\Delta_{\mathbb{H}^n}}\phi d\xi =\int_\Omega a(\xi)u\phi d\xi -\int_\Omega f(\xi,u)\phi d\xi ,~~~~~~~\forall~\phi\in C_c^2(\Omega),
\end{equation}
where $C_c^2(\Omega)$ is the space of $C^2$ functions in $\Omega$ with a compact support in $\Omega.$
The linearized operator $L_u$ associated with \eqref{pm} at a given solution $u$ is defined by the following duality:
\[ L_u: v\in D^{2, 2}(\Omega)\cap D_0^{1, 2}(\Omega)  \longrightarrow L_{u}(v)\in(D^{2, 2}(\Omega)\cap D_0^{1, 2}(\Omega) )',\]
where
\[L_{u}(v): \psi \in D^{2, 2}(\Omega)\cap D_0^{1, 2}(\Omega)  \longrightarrow L_{u}(v,\psi)\]
and
\[L_u(v,\psi)= \int_\Omega{\Delta_{\mathbb{H}^n}} v {\Delta_{\mathbb{H}^n}}\psi d\xi -\int_\Omega a(\xi)v\psi d\xi +\int_\Omega f_{u}(\xi,u)v \psi d\xi .\]

It is easy to see that $L_u$ is well-defined and the first eigenvalue of $L_u$ is given by
\be\label{pev}
\la_1= \inf_{\substack{v\in D^{2, 2}(\Omega)\cap D_{0}^{1,2}(\Om)\\v\neq 0}} \frac{L_{u}(v,\,v)} {\int_{\Om}v^2 d\xi  }.
\ee

We say that the solution $u$ of \eqref{pm} is semi-stable if
\begin{equation}\label{stb}
\int_\Omega|{\Delta_{\mathbb{H}^n}} v|^2 d\xi -\int_\Omega a(\xi)v^2 d\xi +\int_\Omega f_{u}(\xi,\,u) v^2 d\xi  \geq  0
\end{equation}
for every $v\in~C_c^2(\Omega),$ see \cite{wei2} and the references therein for the definition of stability (semi-stability) of solutions to biharmonic problems.
Actually, \eqref{stb} implies that the principal eigenvalue of the linearized equation
associated with \eqref{pm} is non-negative and hence the solution $u$ of \eqref{pm}  is semi-stable.

This work is motivated by the recent works on polyharmonic equations, where the authors obtained 
stability properties of solution to polyharmonic equation with exponential nonlinearity, see \cite{farina,huang} and for the stability results to biharmonic equation, 
see \cite{ber,davila,pkara}
and for Liouville theorems of stable radial solution to biharmonic equations, see \cite{war}. The nonlinearities of the type
\[a(\xi)u-f(\xi,u)\]
are known as logistic-type nonlinearities, see for instance \cite{af,canada,du, gao,take,take1} and references therein, which motivate us to 
consider the same for the stability questions to biharmonic equations on Heisenberg group.

In the context of the above research works, it is natural to ask, whether we can obtain stability/semi-stability of the positive solution to biharmonic problems with 
logistic-type nonlinearities on Heisenberg group. More precisely, the aim of this paper is to answer this question. 

In fact, first,  we establish the existence of positive solution to \eqref{pm} and then we prove the semi-stability of the positive solution to \eqref{pm}. 
We establish the existence of a solution by Schauder's fixed point theorem and then strong maximum principle yields the positivity of the solution.
In order to show the semi-stability of the solution, we show that the principal eigenvalue of the linearized equation
associated with \eqref{pm} is non-negative and hence the solution $u$ of \eqref{pm} is semi-stable. This is shown by testing the equation under consideration against a suitable 
test function.

To the best of our knowledge, we are not aware of any results on the semi-stability of positive solution for the biharmonic equations on the Heisenberg group. 
For the existence of positive solution to problems similar to \eqref{pm} in $\Omega \subseteq \R^n,$  we refer to \cite{bernis,dal,davila1,davila2,ebo, sat,zhang2,zhanglu} 
and the references therein.  For the existence of positive solution to equations involving  Kohn-Laplace operator on Heisenberg group, we refer to \cite{bran,zhang,zhang1} and 
references cited therein.\\
We make the following hypotheses on the nonlinearity $f$ and weight $a$:

\noindent(H1) Let $f\in C(\overline{\Omega}\times\mathbb{R}, \mathbb{R})$ and $C^{1}$ in the $y$ variable and be satisfy
\[f_y(\xi,y)\geq\frac{f(\xi,y)}{y},\,\,\,\,\,\forall\,\,\, 0< y\in\mathbb{R},\,\,\,\forall\,\,\xi\in \Om.\]

\noindent(H2) Let $a\in L^{\infty}(\Omega).$ There exists $r\in L^{q'}(\Omega)$ such that
\begin{enumerate}[(i)]
    \item for $n=1(Q=4),$  $|f(\xi,s)| \leq r(\xi)+c|s|^{q-1},\,\text{ for a.e.}\, \xi\in\Om, \, \forall s\in \mathbb{R},$ where $r\in L^q(\Om),$ and  $c>0$ is a 
    constant  such that $(c+\norm{a}_\infty)c_{emb}^2 < 1$ and $c_{emb}$  is constant of the embedding of $D^{2, 2}(\Omega)\cap D_0^{1, 2}(\Omega)$ into  $L^q(\Om),\, q\geq 1$ is arbitrary.
\item for $n> 1 (Q>4),$  $|f(\xi,s)| \leq r(\xi)+c|s|^{q},\,\text{ for a.e.}\, \xi\in\Om, \, \forall s\in \mathbb{R},$ where $q<\frac{Q+4}{Q-4},$  $r\in L^q(\Om),$ and  $c>0$ 
is a constant  such that $(c+\norm{a}_\infty)c_{emb}^2 < 1.$
\end{enumerate}

\noindent(H3) Let $a(\xi)s-f(\xi,s) \geq 0,$ for a.e. $\xi\in \Omega$ and for all $s\in \R.$

\begin{rem}
The functions of the form $f(\xi,y)=g(\xi)y^r,$ where  $r>1$ and $g(\xi)\geq 0$ satisfy  hypothesis (H1).
\end{rem}
The functional associated with \eqref{pm} is
$$E\colon D^{2, 2}(\Omega)\cap D_0^{1, 2}(\Omega)  \longrightarrow \mathbb{R}$$
defined by
\[E(u)=\frac{1}{2}\int_\Omega|{\Delta_{\mathbb{H}^n}} u|^2 d\xi-\frac{1}{2}\int_\Omega a(\xi)u^2 d\xi +\int_\Omega F(\xi,\,u)d\xi,\]
where
\[F(\xi,s)=\int_0^s f(\xi,t)dt.\]
Throughout the article, the space $D^{2, 2}(\Om)\cap D_0^{1,2}(\Om)$ is denoted by $D.$

The following are the main results of this paper, which we will prove in the last section.
\begin{theorem}\label{th2}
Let \rm{(H2)}--\rm{(H3)} hold. Then \eqref{pm} has a positive solution.
\end{theorem}
\begin{theorem}\label{th1}
Let \rm{(H1)}--\rm{(H3)} hold. Let $u\in D \cap L^{\infty}(\Omega)$ be a positive solution of \eqref{pm}. Then $u$ is semi-stable.
\end{theorem}

We organize this paper as follows. Section 2 deals with useful preliminaries on the Heisenberg group and important results, which we shall use in next section. 
The proofs of main theorems and auxiliary lemmas are a part of
Section 3. A few remarks are given in Section 4.

\section{Preliminaries}
We begin this section with the briefs on Heisenberg group and the auxiliary results which are used in order to 
prove the main results. The Heisenberg group $\H^{n}= (\R^{2n+1},\,\textbf{.}),$ is the space
$\R^{2n+1}$ with the non-commutative law of product
$$(x,y,t)\textbf{.}(x',y',t')= (x+x',\,y+y', t+t'+ 2(\langle y,x'\rangle-\langle x,\,y'\rangle)),                     $$
where $x,\,y,\,x',\,y'\in \R^{n},\,t,\,t'\,\in \R$ and $\langle \cdot,\cdot \rangle$ denotes the standard inner product in $\R^{n}.$ The homogeneous dimension of $\H^n$ is $Q=2n+2.$
This operation endows $\H^{n}$ with the structure of a Lie group. The Lie algebra of $\H^{n}$
is generated by the left-invariant vector fields
$$ T= \frac{\partial}{\partial t},\,\,\,X_{i}= \frac{\partial}{\partial x_{i} }+ 2y_{i} \frac{\partial}{\partial t},\,\,
Y_{i}= \frac{\partial}{\partial y_{i} }- 2x_{i} \frac{\partial}{\partial t},\,\,i=1,\,2,.\,3,\,\ldots,\,n. $$
These generators satisfy the non-commutative formula
$$[X_{i},Y_{j}]= -4 \delta_{ij} T,\,\,[X_{i},X_{j}]= [Y_{i},Y_{j}]=  [X_{i},T] = [Y_{i},T] =0.$$

Let
$ z=(x,\,y)\in \R^{2n},\,\,\xi= (z,\,t)\in \H^{n}.$ The parabolic dilation $$\delta_{\la}\xi= (\la x,\,\la y,\,\la^{2} t)$$ satisfies
$$  \delta_{\la}(\xi_{0}\textbf{.}\xi     )  =  \delta_{\la}\xi   \textbf{.} \delta_{\la} \xi_{0}  $$
and $$ || \xi||_{\H^{n}}  = (|z|^{4}+ t^{2} )^{\frac{1}{4} }= ((x^2 + y^2)^{2}+ t^{2} )^{\frac{1}{4} }  $$
is a norm with respect to the parabolic dilation which is known  as Kor\'{a}nyi gauge norm $N(z,\,t).$
In other words,
$\rho(\xi) = (|z|^{4}+ t^{2} )^{\frac{1}{4} }$
 denotes the Heisenberg distance between $\xi$ and the origin.
Similarly, one can define the distance between
$(z,\,t)$ and $(z',\,t')$ on $\H^{n}$ as follows:
$$\rho(z,\,t;z',t')= \rho((z',\,t')^{-1}\,.\,(z,\,t)).  $$
 It is clear that the vector fields $X_{i},\,Y_{i},\,i=1,\,2,\,\ldots,\,n$ are homogeneous of degree $1$ under the norm
$||.\,||_{\H^{n}}$ and $T$ is homogeneous of degree $2.$
The Kor\'{a}nyi ball of center $\xi_{0}$ and radius $r$ is defined by
$$  B_{\H^{n}}(\xi_{0},\,r)= \{ \xi: ||\xi^{-1}\textbf{.} \xi_{0} || \leq r  \}                  $$
and it satisfies
$$ |B_{\H^{n}}(\xi_{0},\,r)| = |B_{\H^{n}}(0,\,r)| =  r^{Q}|B_{\H^{n}}(\textbf{0},\,1)|,            $$
where $|.|$ is the $(2n+1)$-dimensional Lebesgue measure on $\H^{n}$  and $Q= 2n+2$ is the so-called the homogeneous dimension of Heisenberg group $\H^{n}.$
The Heisenberg gradient and Heisenberg Laplacian or the Laplacian-Kohn operator on
$\H^{n}$ are defined as follows:
$$\nabla_{\H^{n}}= (X_{1},\,X_{2},\,\ldots,\,X_{n},\, Y_{1},\,Y_{2},\,\ldots,\,Y_{n})$$
and
$$  \De_{\H^{n}}= \sum_{i=1}^{n} X_{i}^{2}+ Y_{i}^{2}=
\sum_{i=1}^{n}\left(\frac{\partial^{2} }{\partial x_{i}^{2}}+  \frac{\partial^{2} }{\partial y_{i}^{2}} +
4 y_{i} \frac{\partial^{2} }{\partial x_{i}\partial t}- 4 x_{i} \frac{\partial^{2} }{\partial y_{i}\partial t}
+ 4(x_{i}^{2}+y_{i}^{2} )\frac{\partial^{2} }{\partial t^{2}}   \right).$$

\begin{definition}[$D^{1,p}(\Omega)$ and $D_0^{1,p}(\Omega)$ Space]
 Let $\Omega \subseteq \h$ be open and $1<p<\infty$. Then we define
 \[D^{1,p}(\Omega)=\{u:\Omega\rightarrow \R\text{ such that }\, u,|\nh u| \in L^p(\Omega) \}.\]
 $D^{1,p}(\Omega)$ is equipped with the norm
 \[\norm{u}_{D^{1,p}(\Omega)}=\left(\norm{u}_{L^p(\Omega)}+\norm{\nh u}_{L^p(\Omega)}\right)^{\frac{1}{p}}.\]
 $D_0^{1,p}(\Omega)$ is the closure of $C_c^\infty(\Omega)$ with respect to the norm
 \[\norm{u}_{D_0^{1,p}(\Omega)}=\left(\int_\Omega |\nh u|^p dz dt\right)^{\frac{1}{p}}.\]
\end{definition}
\begin{definition}[$D^{2,p}(\Omega)$ and $D_0^{2,p}(\Omega)$ Space]
 Let $\Omega \subseteq \h$ be open and $1<p<\infty$. Then we define
 \[D^{2,p}(\Omega)=\{u:\Omega\rightarrow \R\text{ such that }\, u,|\nh u|, |\Dh u| \in L^p(\Omega) \}.\]
 $D^{2,p}(\Omega)$ is equipped with the norm
 \[\norm{u}_{D^{2,p}(\Omega)}=\left(\norm{u}_{L^p(\Omega)}+\norm{\nh u}_{L^p(\Omega)}+\norm{\Dh u}^p\right)^{\frac{1}{p}}.\]
 $D_0^{2,p}(\Omega)$ is the closure of $C_c^\infty(\Omega)$ with respect to the norm
 \[\norm{u}_{D_0^{2,p}(\Omega)}=\left(\int_\Omega |\Dh u|^p dz dt\right)^{\frac{1}{p}}.\]
\end{definition}
\begin{theorem}[Compact Embedding \cite{bong}]\label{emb} Let $\Om$ be a bounded domain in $\H^n$ and $Q=2n+2$ be homogeneous dimension of $\H^n,$ then the following embeddings are compact:
\begin{enumerate}[(i)]
\item If $Q=4,$ then $D^{2, 2}(\Omega) \cap D_0^{1,2}(\Omega)\hookrightarrow L^q(\Om),  \,\, 1\leq q<\infty.$
\item If $Q>4,$ then $D^{2,2}(\Omega) \cap D_0^{1,2}(\Omega)\hookrightarrow L^q(\Om),  \,\, 1\leq q<\frac{2Q}{Q-4}.$
\end{enumerate}
\end{theorem}

\begin{theorem}[Nemytski operator \cite{dra}]\label{nem} Let $\Om\subseteq H^n$ and  $f:\Omega\times \R \rar \R$ be a Carath\'{e}odory function and $p,q\in \left[\left.1,\infty\right)\right.$ Let there exist $g\in L^q(\Omega)$
and $c\in \R$ be such that
\[|f(\xi, y) \leq g(\xi)+c|y|^{p/q},\,\,\text{ for a.e. } \xi\in \Omega \text{ and for all } y\in \R.\]
Then the operator $F: \phi \mapsto f(\cdot, \phi(\cdot)) $ has the following properties:
\begin{enumerate}[(i)]
\item $F(\phi)\in L^q(\Om)$ for all $\phi\in L^p(\Om).$
\item  $F$ is a continuous mapping from $L^p(\Om)$ into $L^q(\Om).$
\item $F$ maps bounded sets in $L^p(\Omega)$ into bounded sets in $L^q(\Omega)$.
\end{enumerate}
\end{theorem}
\begin{theorem}[Schauder fixed point theorem \cite{dra}]\label{schauder}
Let $\mathcal{M}$ be a nonempty, closed,
convex and bounded subset of a normed linear space $X$. Assume that $F$ is a compact operator from $\mathcal{M}$ to X such that
and $F(\mathcal{M})\subseteq \mathcal{M}$. Then there is a fixed point of $F$ in $\mathcal{M}$.
\end{theorem}
\section{Proof of Theorem \ref{th2} and Theorem \ref{th1}:}
In order to prove Theorem \ref{th1}, first, we prove the following lemma:
\begin{lem}\label{l1}
Let $u\in D^{2,2}(\Om)\cap D_{0}^{1,2}(\Om)$ be a nonnegative weak solution (not identically zero) of
 \begin{equation}\label{ev1}
\Dh^{2} u =  a(\xi) u-f(\xi,u)\,\,\,\mbox{in}\,\,\Omega,\,\,\,\,u= \Dh u = 0\,\,\,\mbox{on}\,\,\partial\Om,
\end{equation}
where $a$ and $f$ satisfy \rm{(H3)}, then $-\Dh u>0$ in $\Om$ and $u>0$ in $\Om.$
\end{lem}

\begin{proof}
 Let $-\Dh u =v.$ Then writing \eqref{ev1} into system form, we get
\begin{equation}\label{sy}
\left\{
\begin{array}{ll}
    -\Dh u = v\,\,\,\mbox{in}\,\,\,\, \Omega,\\
    -\Dh v= a(\xi) u-f(\xi,u)\,\,\,\mbox{in}\,\,\,\,\Omega,\\
   u=0=v\,\,\,\,\,\mbox{on} \,\,\,\partial\Omega.
  \end{array}
\right.
\end{equation}
Since $a(\xi)u-f(\xi,u)\geq 0$ in $\Om,$ so by maximum principle \cite{bony}, we get $v\geq 0.$ By strong maximum principle, either $v>0$ or $v\equiv 0$ in $\Om.$
If $v\equiv 0,$ then we have
$$ -\Dh u = 0\,\,\,\mbox{in}\,\,\,\, \Omega;\,\,\,u= 0\,\,\,\mbox{on} \,\,\partial \Om.$$
Again by maximum principle, we get $u\equiv 0,$ which is a contradiction and therefore $v>0$ in $\Om$ and hence
$$-\Dh u > 0\,\,\,\mbox{in}\,\, \Omega.  $$
Again, since $-\Dh u>0$ in $\Om,$ by strong maximum principle, we get
$$u>0 \,\,\text{in } \Omega.$$
 \end{proof}

 \begin{lem}\label{re1}
  Let \rm{(H2)}--\rm{(H3)} hold. Let $u\in D^{2,2}(\Om)\cap D_{0}^{1,2}(\Om)$ be a nonnegative weak solution (not identically zero) of \eqref{ev1}.  
  Then $u\in C^{4,\alpha}(\Omega).$
 \end{lem}
 \begin{proof}
 Let $-\Dh u =v.$ Then writing \eqref{ev1} into system form, we get
\begin{equation}\label{sy1}
\left\{
\begin{array}{ll}
    -\Dh u = v\,\,\,\mbox{in}\,\,\,\, \Omega,\\
    -\Dh v= a(\xi) u-f(\xi,u)\,\,\,\mbox{in}\,\,\,\,\Omega,\\
   u=0=v\,\,\,\,\,\mbox{on} \,\,\,\partial\Omega.
  \end{array}
\right.
\end{equation}
An application of Theorem\,\ref{emb} with boot-strap arguments yields that $v\in L^{\infty}(\Omega).$
Now, by using Theorem 3.35 \cite{cap} for second equation in \eqref{sy1}, we conclude that $v\in C^\alpha(\Omega)$ for some $0<\alpha<1.$ Then by using Theorem 3.9 \cite{xu},
we get that $u\in C^{2,\alpha}(\Omega).$
Again applying Theorem 3.35 \cite{cap} and Theorem 3.9 \cite{xu} for $u\in C^{2,\alpha}(\Omega),$ we conclude that $u\in C^{4,\alpha}(\Omega).$ This completes the proof of this lemma.
\end{proof}

\noindent{\textbf{Proof of Theorem \ref{th2}}:} The weak formulation of \eqref{pm} is as follows:

\noindent Find $u\in D$ such that
\be\label{weakform}
\int_\Om \Dh u \Dh \phi d\xi =\int_\Om a(\xi) u \phi d\xi-\int_\Om f(\xi,u) \phi d\xi,\,\,\, \forall \,\phi\in D.
\ee
For each $u\in D,$ let us define
\[\bar{L}_u:v \mapsto \int_\Om \Dh u\Dh v d\xi\]
 and
 \[\bar{S}_u: v\mapsto \int_\Om (a(\xi)u-f(\xi,u))v d\xi.\]
Since $\bar{L}_u$ and $\bar{S}_u$ are continuous linear functionals on Hilbert space $D,$ therefore by Reisz representation theorem, there exist unique elements $Lu$ and $Su$ such that
\be\label{w1}
\bar{L}_u(v) =\langle Lu, v\rangle,\,\, \bar{S}_u(v) =\langle Su, v \rangle,\,\,\forall v\in D.
\ee
In view of \eqref{w1}, solving \eqref{weakform} is equivalent to showing that the operator equation
\be\label{w2}
Lu=Su
\ee
has a solution in $D.$ Observe that, the  operator Equation \eqref{w2} has a solution if the operator $S:D\rar D$ has a fixed point. We will use Schauder fixed point theorem to prove that the operator $S,$ in fact, has a fixed point. First, we show that $S$ is compact.
 Let $\mathcal{M} \subseteq D$ be a bounded subset and $\{u_n\}$ be a sequence in $S(\mathcal{M})$. Then there exists $\{w_n\}\subseteq \mathcal{M}$ such that
\[S(w_n)=u_n.\]
Since $D$ is a reflexive Banach space and $\{w_n\}$ is a bounded sequence in $D,$ therefore, up to a subsequence
   \[w_n \rightharpoonup w \text{ in } D.\]
Next, our aim is to show that $S(w_n)$ converges to $S(w)$ for some $w\in \mathcal{M}.$ Consider

\begin{align}\nonumber
\norm{S(w_n)-S(w)}&= \sup_{{\substack{v\in D \\ \norm{v}\leq 1}}}\left|\langle S(w_n)-S(w), v\rangle \right|\\ \nonumber
& \leq  \sup_{{\substack{v\in D \\ \norm{v}\leq 1}}}\left( \int_\Om |a(\xi)||w_n-w||v| d\xi +\int_\Om |f(\xi,w_n)-f(\xi,w)||v| d\xi\right) \\ \nonumber
& \leq \sup_{{\substack{v\in D \\ \norm{v}\leq 1}}}\left(\norm{a}_\infty\int_\Om |w_n-w||v| d\xi +\int_\Om |f(\xi,w_n)-f(\xi,w)||v| d\xi\right)\\ \label{c1}
&\leq \sup_{{\substack{v\in D \\ \norm{v}\leq 1}}}\left(\norm{a}_\infty\norm{w_n-w}_{q}\norm{v}_{q'} +\norm{f(\cdot,w_n)-f(\cdot,w)}_{q'}\norm{v}_{q} \right),
\end{align}
where
\begin{enumerate}[(i)]
\item $1\leq q< \infty$ and $q'=\frac{q}{q-1},$ when $Q=4.$
\item $1\leq q<\frac{2Q}{q-4}$ and $q'=\frac{q}{q-1},$ when $Q>4.$
\end{enumerate}
By Theorem \ref{nem}, the Nemytski operator is continuous from $L^q(\Om)$ to $L^{q'}(\Om),$ thus
\be\label{nc}
\norm{f(\cdot,w_n)-f(\cdot,w)}_{q'}\rar 0 \text{ as } n\rar \infty.
\ee
Also,  by Theorem \ref{emb}, $D $ is compactly embedded into $L^q(\Om),$ thus
\be\label{q5}
w_n \longrightarrow w \text{ in } L^q(\Om).
\ee
On using \eqref{nc} and \eqref{q5} in \eqref{c1}, we get
\[\norm{S(w_n)-S(w)} \longrightarrow 0 \text{ as } n\rar \infty.\]
This shows that  $S$ is a compact operator. Now, it remains to show that $S$ maps $\overline{B(0,R)}$ into $\overline{B(0,R)}.$
Consider
\begin{align*}
\norm{Su}&=\sup_{{\substack {v\in D\\ \norm{v}\leq 1}}}\left|\langle Su, v\rangle\right|\\
&\leq \sup_{{\substack {v\in D\\ \norm{v}\leq 1}}} \left|\int_\Om a(\xi) u v d\xi -\int_\Om f(\xi,u) v d\xi \right|\\
& \leq \sup_{{\substack {v\in D\\ \norm{v}\leq 1}}} \left(\int_\Om |a(\xi)||uv| d\xi+\int_\Om |f(\xi,u)||v| d\xi\right)\\
& \leq \sup_{{\substack{v\in D \\ \norm{v}\leq 1}}}\left(\norm{a}_\infty\int_\Om |u||v| d\xi +\int_\Om |f(\xi,u)||v| d\xi \right)\\
&\leq \sup_{{\substack{v\in D \\ \norm{v}\leq 1}}}\left(\norm{a}_\infty\norm{u}_{q'}\norm{v}_{q} +\int_\Om |r(\xi)+cu^{q-1}||v| \right)\,\, \left(\text{by (H3)}\right)\\
&\leq \sup_{{\substack{v\in D \\ \norm{v}\leq 1}}}\left(\norm{a}_\infty \norm{u}_{q'}\norm{v}_{q}+(\norm{r}_{q'}+c\norm{u}_{q'})\norm{v}_q\right)\\
&\leq  c_{emb}^2 \norm{a}_\infty \norm{u}+cc_{emb}^2\norm{u}+c_{emb}\norm{r}_{q'} \,\,\,\text{(by Theorem \ref{emb})},
\end{align*}
where
\begin{enumerate}[(i)]
\item $1\leq q< \infty$ and $q'=\frac{q}{q-1},$ when $Q=4.$
\item $1\leq q<\frac{2Q}{q-4}$ and $q'=\frac{q}{q-1},$ when $Q>4.$
\end{enumerate}
Since $(c+\norm{a}_\infty)c_{emb}^2<1,$ we choose $R>0$ such that
\[c_{emb}\norm{r}_{q'}+(c+\norm{a}_\infty)c_{emb}^2 R<R,\]
 that is,
 \[R>\frac{c_{emb}\norm{r}_{q'}}{1-(c+\norm{a}_\infty)c_{emb}^2},\]
  then $S$ maps $\overline{B(0,R)}$ into $\overline{B(0,R)}.$
Therefore, by Scahuder's fixed point (Theorem\,\ref{schauder}), $S$ has a fixed point and hence \eqref{pm} has a weak solution. 
 An application of  Lemma \ref{l1}, yields that   $u>0$ in $\Omega.$\\\\
 \qed

\noindent\textbf{Proof of Theorem \ref{th1}:} By Theorem\,\ref{th2}, there exists a positive solution $u\in D\cap L^\infty (\Omega)$ of \eqref{pm}.  By Lemma\,\ref{re1}, we get $u\in C^{4,\alpha}(\Omega).$
Now, for any $v\in C_c^2(\Omega),$ we choose
\[\phi=\frac{v^2}{u}\]
as a test function in \eqref{we}. Since
\[\nh  \phi =\frac{2uv\nh  v-v^2\nh  u}{u^2},\] and
\[{\Delta_{\mathbb{H}^n}} \phi=\frac{2u^3 |\nh  v|^2-4 v u^2 \nh  u.\nh  v+2v^2 u |\nh  u|^2+ 2vu^3 {\Delta_{\mathbb{H}^n}} v- v^2u^2 {\Delta_{\mathbb{H}^n}} u}{u^4}\]
so from \eqref{we}, we get
\begin{align*}\int_\Omega{\Delta_{\mathbb{H}^n}} u.\left[\frac{2u^3 |\nh  v|^2-4 v u^2 \nh  u.\nh  v+2v^2 u |\nh  u|^2+ 2vu^3 {\Delta_{\mathbb{H}^n}} v- v^2u^2 {\Delta_{\mathbb{H}^n}} u}{u^4}\right] d\xi  &\\
  = \int_\Omega a(\xi)v^2d\xi -\int_\Omega \frac{f(\xi,u)v^2}{u}d\xi .
\end{align*}
This yields that
\begin{align*}
& \int_\Omega \frac{-4v}{u^2}{\Delta_{\mathbb{H}^n}} u\nh  u.\nh  vdx+ \int_\Omega \frac{2v}{u} {\Delta_{\mathbb{H}^n}} u{\Delta_{\mathbb{H}^n}} vdx+\int_\Omega \frac{2}{u}|\nh  v|^2 {\Delta_{\mathbb{H}^n}} udx -
\int_\Omega\frac{v^2}{u^2}|{\Delta_{\mathbb{H}^n}} u|^2d\xi     \\
& + \int_\Omega \frac{2 v^2}{u^3}|\nh u|^2{\Delta_{\mathbb{H}^n}} udx -\int_\Omega a(\xi) v^2d\xi +{\int_\Omega f(\xi,u)\frac{v^2}{u}d\xi }+{\int_\Omega |{\Delta_{\mathbb{H}^n}} v|^2 d\xi }-\int_\Omega |{\Delta_{\mathbb{H}^n}} v|^2d\xi \\
& =0.
\end{align*}
On re-arranging the terms, we get
\begin{align*}
 \int_\Omega |{\Delta_{\mathbb{H}^n}} v|^2d\xi &-\int_\Omega a(\xi)v^2d\xi + \int_\Omega \frac{f(\xi,u)}{u}v^2d\xi  =\int_\Omega  \Big[~{|{\Delta_{\mathbb{H}^n}} v|^2}+\frac{4v}{u^2}{\Delta_{\mathbb{H}^n}} u \nh  u.\nh  v\\
& -{\frac{2v}{v}{\Delta_{\mathbb{H}^n}} u{\Delta_{\mathbb{H}^n}} v} -\frac{2}{u}|\nh  v|^2 {\Delta_{\mathbb{H}^n}} u +{\frac{v^2}{u^2}|{\Delta_{\mathbb{H}^n}} u|^2}-\frac{2v^2}{u^3}|\nh  u|^2{\Delta_{\mathbb{H}^n}} u\Big]d\xi   \\
&=\int_\Omega \Big[\Big({\Delta_{\mathbb{H}^n}} v-\frac{v}{u}{\Delta_{\mathbb{H}^n}} u\Big)^2 + 
\frac{4v}{u^2}{\Delta_{\mathbb{H}^n}} u \nh  u\nh  v- \frac{2}{u}|\nh  v|^2{\Delta_{\mathbb{H}^n}} u\\
&-\frac{2v^2}{u^3}|\nh  u|^2 {\Delta_{\mathbb{H}^n}} u \Big]d\xi .
\end{align*}
This implies that
\begin{align*}
& \int_\Omega |{\Delta_{\mathbb{H}^n}} v|^2d\xi -\int_\Omega a(\xi)v^2d\xi +\int_\Omega \frac{f(\xi,u)}{u}v^2d\xi \\
&\geq
\int_\Omega \Big[\frac{4v}{u^2}{\Delta_{\mathbb{H}^n}} u \nh  u\nh  v - \frac{2}{u}|\nh  v|^2{\Delta_{\mathbb{H}^n}} u-
\frac{2v^2}{u^3}|\nh  u|^2 {\Delta_{\mathbb{H}^n}} u \Big]d\xi   \\
&= \int_\Omega \Big[ -\frac{2}{u}{\Delta_{\mathbb{H}^n}} u \Big( |\nh  v|^2+\frac{v^2}{u^2}|\nh  u|^2-\frac{2v \nh  u\nh  v  }{u}\Big)\Big]d\xi   \\
&= \int_\Omega \Big(-\frac{2}{u}{\Delta_{\mathbb{H}^n}} u\Big)\Big(\nh  v-\frac{v}{u}\nh  u\Big)^2d\xi \\
&\geq 0\,\, (\mbox{by Lemma}\,\ref{l1}).
\end{align*}
This gives us 
\begin{align*}
\int_\Omega |{\Delta_{\mathbb{H}^n}} v|^2d\xi -\int_\Omega a(\xi)v^2d\xi +\int_\Omega \frac{f(\xi,u)}{u}v^2d\xi  \geq 0.
\end{align*}
Now, using \rm{(H1)}, we obtain
\begin{equation}
\int_\Omega |\Dh v|^2d\xi -\int_\Omega a(\xi)v^2d\xi +\int_\Omega f_{u}(x,u)v^2d\xi \geq 0
\end{equation}
and therefore $u$ is semi-stable. This completes the proof of this theorem.

\section{ A few remarks}
The following remarks are in order:
\begin{rem}
In fact, one can also consider the stability questions to the following class of problems
\begin{equation}\label{rg1}
\left\{
  \begin{array}{ll}
    \Delta_{\mathbb{H}^n}^2u=f_1(\xi,u)\,\,\,\,\text{in}\,\,\,\Omega\\
    u>0 \,\,\,\,\,\, \text{ in } \Omega\\
   u=0=\Delta u \,\,\,\,\, \text{on } \partial\Omega,
 \end{array}
\right.
\end{equation}
where $\Omega\subset\mathbb{H}^n$ is an open, smooth and bounded subset,  $f_{1}\in C(\Omega\times \R,\,\R)$ and $f_{1}(\xi, u)\geq a(\xi) u- f(\xi, u),$ 
where $a$ and $f$ satisfy the same hypotheses as above. More precisely, one can have the following theorem, which is an easy adaptaion of Theorem\,\ref{th1}:
\begin{theorem}
Let $f_{1}\in C(\Omega\times \R,\,\R)$ and $f_{1}(\xi, u)\geq a(\xi) u- f(\xi, u).$  Let \rm{(H1)}--\rm{(H3)} hold. 
Let $u\in D \cap L^{\infty}(\Omega)$ be a positive solution of \eqref{rg1}. Then $u$ is semi-stable.
\end{theorem}

\end{rem}

In the next remark, let us consider the biharmonic boundary value problem posed in an euclidean ball but with Dirichlet's boundary conditions.

\begin{rem}\label{thd2}
Let us consider the following problem in the euclidean domains:
\begin{equation}\label{pd}
\left\{
\begin{array}{ll}
    \Delta^2 u=a(\xi)u-f(\xi,\,u)\,\,\,\mbox{in}\,\,\,\, B,\\
   u=0=\frac{\partial u}{\partial \nu}\,\,\,\,\,\mbox{on} \,\,\,\partial B,
  \end{array}
\right.
\end{equation}
where $B$ denotes unit ball in $\mathbb{R}^n,$ $a\in L^{\infty}(B)$ and $f\in C(\overline{B}\times \R,\,\R).$ Let (H2)--(H3) hold. 
Then, using the similar lines of proof as in Theorem \ref{th2} and Boggio maximum principle \cite{gazz}, one can prove that \eqref{pd} 
has a positive solution and using the similar lines of proof as in Theorem \ref{th1},  it is easy to see that $u$ is semi-stable. For the sake of brevity, we skip the details.
\end{rem}

\begin{rem}
In the context of Remark\,\ref{thd2}, due to the lack of the positivity of Green's function for any arbitrary smooth bounded domain, it would be of interest to establish the existence of
positive solution and its semi-stability
of \eqref{pd} in any arbitrary smooth, bounded domain $\Omega\subset \mathbb{R}^n,$ and in more general, when $\Omega\subset\mathbb{H}^n.$
\end{rem}

\vskip .002in

\begin{center}
{\bf Acknowledgment} \end{center}
The first author thanks DST/SERB for the financial support under the grant MTR/2018/000233 and the second author thanks DST/SERB for the financial support under the grant MTR/2018/000096.

\end{document}